\documentclass[a4paper,12pt]{amsart}

\usepackage{amsmath,amsfonts,amsthm,amssymb,amscd,enumerate,esint,vmargin}
\usepackage[bookmarks, colorlinks=true]{hyperref}
\usepackage{tikz}
\usepackage{xypic}

\newcommand{\CC}{\mathbb{C}}
\newcommand{\NN}{\mathbb{N}}
\newcommand{\RR}{\mathbb{R}}
\newcommand{\ZZ}{\mathbb{Z}}
\newcommand{\mc}{\mathcal}

\newcommand{\map}[3]{#1 \colon #2 \rightarrow #3}
\newcommand{\norm}[1]{\| #1 \|}

\newcommand{\supp}{\operatorname{supp}}

\newcommand{\sm}{\setminus}

\newcommand{\wtd}{\widetilde}

\newcommand{\floor}[1]{\lfloor #1 \rfloor} 
\newcommand{\ceil}[1]{\lceil #1 \rceil}

\theoremstyle{theorem}
\newtheorem{theorem}{Theorem}[section]
\theoremstyle{definition}
\newtheorem{definition}[theorem]{Definition}
\newtheorem{example}[theorem]{Example}
\theoremstyle{remark}
\newtheorem{remark}[theorem]{Remark}
\theoremstyle{plain}
\newtheorem{lemma}[theorem]{Lemma}
\theoremstyle{plain}
\newtheorem{corollary}[theorem]{Corollary}
\newtheorem{proposition}[theorem]{Proposition}
\theoremstyle{definition}

\title[Manifolds with $L^p$-unbounded Riesz transform]{New Riemannian manifolds with $L^p$-unbounded Riesz transform for $p > 2$}
\author{Alex Amenta}
\address{Mathematisches Institut \\ Universit\"at Bonn, Germany}
\email{amenta@math.uni-bonn.de}

\begin{document}
\maketitle

\begin{abstract}
  We construct a large class of Riemannian manifolds of arbitrary dimension with Riesz transform unbounded on $L^p(M)$ for all  $p > 2$.
  This extends recent results for Vicsek manifolds, and in particular shows that fractal structure is not necessary for this property.
\end{abstract}

\section{Introduction}

Consider a Riemannian manifold $M$ with gradient $\nabla$ and Laplace--Beltrami operator $\Delta$.
The Riesz transform $\nabla \Delta^{-1/2}$, with $\Delta^{-1/2}$ defined via the spectral theorem, maps $L^2(M)$ boundedly to the space of square integrable vector fields $L^2(M;TM)$.
Much attention has been given to the question of whether this operator extends to a bounded map from $L^p(M)$ to $L^p(M;TM)$ for $p \neq 2$, or equivalently, whether the estimate
\begin{equation*}
	(R_p): \quad \norm{|\nabla f|}_p \lesssim \norm{\Delta^{1/2} f}_p \qquad \text{for all $f \in C_c^\infty(M)$}
\end{equation*}
holds.
It is conjectured that for $p \in (1,2)$ the estimate $(R_p)$ holds whenever $M$ is complete, with implicit constant depending only on $p$; in \cite{AT19} it is shown that the failure of this uniformity among manifolds of a fixed dimension would imply the existence of a manifold for which $(R_p)$ fails. 

One is naturally led to consider also the `reverse' estimate
\begin{equation*}
	(RR_p): \quad \norm{\Delta^{1/2} f}_p \lesssim \norm{|\nabla f|}_p \qquad \text{for all $f \in C_c^\infty(M)$}.
\end{equation*}
A duality argument shows that for $p \in (1,\infty)$, $(R_p)$ implies $(RR_{p^\prime})$, where $p^\prime = p/(p-1)$ is the H\"older conjugate of $p$.
If $(R_p)$ and $(RR_p)$ both hold, then we have a norm equivalence
\begin{equation*}
	\norm{|\nabla f|}_p \simeq \norm{\Delta^{1/2} f}_p,
\end{equation*}
which says that the homogeneous Sobolev space $\dot{W}^p_1(M)$ may be defined either via the gradient or via the square root of the Laplace--Beltrami operator.

Generally $(R_p)$ holds only for some interval of $p \in (1,\infty)$ including $2$, and proving $(R_p)$ presents different difficulties depending on whether $p < 2$ or $p > 2$.
When $1 < p < 2$, $(R_p)$ is known to follow from the volume doubling property and Gaussian or sub-Gaussian heat kernel upper estimates \cite{CD99,CCFR17} (see also \cite{LZ18} for examples which do not satisfy such kernel estimates).
The volume doubling property and an appropriately scaled $L^2$-Poincar\'e inequality imply $(R_p)$ for some $p > 2$ \cite{AC05}.
In \cite{ACDH04} this is linked with gradient estimates for the heat kernel, and in \cite{BF16} the $L^2$-Poincar\'e inequality is replaced by a relative Faber--Krahn inequality and a reverse H\"older inequality.

Some manifolds for which $(R_p)$ fails for some $p > 2$ are known.
If $M$ is the connected sum of two copies of $\RR^n \sm B(0,1)$ with $n \geq 3$ ---or more generally, an $n$-dimensional manifold with at least two (and finitely many) Euclidean ends---$(R_p)$ holds if and only if $p \in (1,n)$ \cite{CD99,CCH06}.
Similar results are known for conical manifolds \cite{hL99} and for $2$-hyperbolic, $p$-parabolic manifolds with at least two ends  \cite{gC17}.

The most relevant examples to this article are Vicsek manifolds, which are `thickenings' of Vicsek graphs (pictured in the $2$-dimensional case in Figure \ref{fig:vicsek}).
The Vicsek graph, being a graphical realisation of a Vicsek fractal, is a `fractal at infinity'.
Locally a Vicsek manifold behaves like Euclidean space (it is, of course, a manifold), but at large scale it behaves like a fractal.
In \cite{CCFR17} it is shown that for a Vicsek manifold of any dimension, $(R_p)$ holds if and only if $p \in (1,2]$.
The result for $p < 2$ is a consequence of volume doubling and sub-Gaussian heat kernel estimates.
The proof that $(R_p)$ fails for $p > 2$ directly uses the definition of the Vicsek graph \cite[Theorem 5.1]{CCFR17}.

In this article we construct a class of manifolds of arbitrary dimension for which $(R_p)$ fails for all $p > 2$.\footnote{In fact, we prove the stronger result that $(RR_q)$ fails for all $q \in (1,2)$.}
These manifolds are thickenings of what we call \emph{spinal graphs}, satisfying generalised dimension conditions defined in terms of the spinal structure along with a polynomial volume lower bound.
The Vicsek graphs satisfy these conditions, but the proof of this exploits their fractal nature.
We construct a large class of non-fractal spinal graphs with the desired dimension conditions and volume lower bounds, thus yielding manifolds of arbitrary dimension with no fractal structure that fail $(R_p)$ for all $p > 2$.

\section*{Notation}

The graphs we consider are non-directed, with at most one edge per pair of vertices, and with no edges from a vertex to itself.
The set of vertices of a graph $G$ is denoted by $V(G)$, and if two vertices $x,y \in V(G)$ are neighbours we write $x \sim y$.
The set of edges of $G$ is denoted by $E(G)$.
For a connected graph $G$ we let $d_G(x,y)$ denote the combinatorial distance between $x$ and $y$, given by the minimum length of a path from $x$ to $y$, and for $x \in V(G)$, $r > 0$ let
\begin{equation*}
  B_G(x,r) := \{y \in V(G) : d_G(x,y) \leq r\}.
\end{equation*}

\section{Spinal graphs}

\begin{definition}\label{dfn:spinalgraph}
  Let $G$ be a connected graph, $\Sigma \subset V(G)$, and let $\map{\pi}{V(G)}{\Sigma}$ be a function such that
  \begin{itemize}
  \item $\pi(x) = x$ for all $x \in \Sigma$,
  \item $\pi^{-1}(x)$ is finite for all $x \in \Sigma$,
  \item if $a,b \in V(G)$ and $\pi(a) \neq \pi(b)$, then every path from $a$ to $b$ contains a subpath from $\pi(a)$ to $\pi(b)$.
  \end{itemize}
  We refer to $(G,\Sigma,\pi)$ as a \emph{spinal graph}, and the set of vertices $\Sigma$ is called the \emph{spine}.
\end{definition}

\begin{remark}
  One could formulate this definition without the finiteness condition, but it will be convenient for us to keep it.
\end{remark}

An example of a spinal graph $(G,\Sigma,\pi)$ is pictured in Figure \ref{fig:spinal-example}.
There the vertices of the spine $\Sigma$ are shaded black, while the other vertices are unshaded; for each vertex $x$, the point $\pi(x)$ is the (uniquely determined) point on $\Sigma$ of minimal distance to $x$.
The dotted lines are \emph{not} edges of $G$; if they were to be added to $G$, then the resulting graph would not be a spinal graph.

To help the reader familiarise themselves with the definition of a spinal graph we prove the following lemma (which will be useful later).

\begin{lemma}\label{lem:neighbour-lemma}
  Let $(G,\Sigma,\pi)$ be a spinal graph, and suppose $a,b \in V(G)$ with $\pi(a) = \pi(b) =: x$.
  Then every minimal path from $a$ to $b$ is contained entirely in $\pi^{-1}(x)$.
\end{lemma}

\begin{proof}
  Suppose this is false.
  Then there exist $a,b \in V(G)$ with $\pi(a) = \pi(b) =: x$ and a path $\gamma$ from $a$ to $b$ of minimal length which passes through a vertex $c$ with $\pi(c) \neq x$.
  Since $\pi(a) \neq \pi(c)$, by the third condition in the definition of a spinal graph, there exists a subpath $\heartsuit$ of $\gamma$ from $\pi(a) = x$ to $\pi^2(c) = \pi(c)$, so we can write $\gamma$ as a concatenation of paths $\gamma = \alpha \ast \heartsuit \ast \beta$, where $\alpha$ is a path from $a$ to $x$ and $\beta$ is a path from $\pi(c)$ to $b$; viewed as a commutative diagram,
  \begin{equation*}
    \xymatrix{
      a \ar[d]_{\alpha} \ar[r]^{\gamma} & b \\
      x \ar[r]_-{\heartsuit} & \pi(c). \ar[u]_{\beta}
      }
    \end{equation*}
    Similarly, there is a subpath $\heartsuit^\prime$ of $\beta$ from $\pi(c)$ to $\pi(b) = x$, and we can write $\beta = \delta \ast \heartsuit^\prime \ast \delta^\prime$ as a concatenation of paths, summarised by the commutative diagram
    \begin{equation*}
      \xymatrix{
        b & x \ar[l]_{\delta'} \\
        \pi(c) \ar[u]^{\beta} \ar[r]_{\delta} & \pi(c). \ar[u]_{\heartsuit'}
      }
    \end{equation*}
    Putting these commutative diagrams together we get
    \begin{equation*}
      \xymatrix{
      a \ar[d]_{\alpha} \ar[r]^{\gamma} & b & x \ar[l]_{\delta'} \\
      x \ar[r]_-{\heartsuit} & \pi(c) \ar[u]_{\beta} \ar[r]_{\delta} & \pi(c) \ar[u]_{\heartsuit'},
      }
    \end{equation*}
    from which we can read that
    \begin{equation*}
      \ell(\gamma) = \ell(\alpha \ast \heartsuit \ast \delta \ast \heartsuit^\prime \ast \delta^\prime) = \ell(\alpha) + \ell(\heartsuit) + \ell(\delta) + \ell(\heartsuit^\prime) + \ell(\delta^\prime)
    \end{equation*}
    where $\ell(\cdot)$ denotes the length of a path.
    Since $x \neq \pi(c)$, the paths $\heartsuit$ and $\heartsuit^\prime$ have positive length, so we find that
    \begin{equation*}
      \ell(\alpha * \delta^\prime) = \ell(\alpha) + \ell(\delta^\prime) < \ell(\gamma).
    \end{equation*}
    Since $\alpha * \delta^\prime$ is a path from $a$ to $b$, this contradicts the assumption that $\gamma$ has minimal length. 
\end{proof}

  \begin{figure}
    \caption{A spinal graph $(G,\Sigma,\pi)$, with a few vertices labeled. The spine $\Sigma \subset V(G)$ consists of the shaded vertices.}
    \label{fig:spinal-example}
    \begin{center}
      \begin{tikzpicture}[scale=0.8]

    \draw [thin] (0,0) -- (1,0);
 	\draw [thin] (1,0) -- (2,0);
	\draw [thin] (2,0) -- (3,0);
 	\draw [thin] (3,0) -- (4,0);
	\draw [thin] (3,0) -- (3,-1);
	\draw [thin] (3,-1) -- (3,-2);
 	\draw [thin] (4,0) -- (5,0.5);
	\draw [thin] (4,0) -- (5,-0.5);
 	\draw [thin] (5,0.5) -- (6,0);
	\draw [thin] (5,-0.5) -- (6,0);	
	\draw [thin] (6,0) -- (7,0);		

	\draw [fill=black] (0,0) circle [radius = 2.5pt];
	\draw [fill=black] (1,0) circle [radius = 2.5pt];
	\draw [fill=black] (2,0) circle [radius = 2.5pt];
	\draw [fill=black] (3,0) circle [radius = 2.5pt];
	\draw [fill=black] (3,-1) circle [radius = 2.5pt];
	\draw [fill=black] (3,-2) circle [radius = 2.5pt];
	\draw [fill=black] (4,0) circle [radius = 2.5pt];
	\draw [fill=black] (5,0.5) circle [radius = 2.5pt];
	\draw [fill=black] (5,-0.5) circle [radius = 2.5pt];	
	\draw [fill=black] (6,0) circle [radius = 2.5pt];
	\draw [fill=black] (7,0) circle [radius = 2.5pt];
	
 	\draw [thin] (1,0) -- (1,0.5);
 	\draw [thin] (1,0.5) -- (0.8,1);
 	\draw [thin] (1,0.5) -- (1.4,1.3);
 	\draw [thin] (0.8,1) -- (1.2,1.7);
 	\draw [thin] (1.4,1.3) -- (1.2,1.7);

 	\draw [thin] (2,0) -- (1.8,0.4);
 	\draw [thin] (1.8,0.4) -- (2.1,0.8);

 	\draw [thin] (3,0.0) -- (3.1,0.6);
 	\draw [thin] (3.1,0.6) -- (3,1.2);

 	\draw [thin] (3,-1) -- (2.4,-1.3);
 	\draw [thin] (2.4,-1.3) -- (1.8,-1.3);
 	\draw [thin] (2.4,-1.3) -- (2,-0.8);
 	\draw [thin] (1.8,-1.3) -- (1.1,-1.2);

 	\draw [thin] (5,-0.5) -- (4.8,-1);
 	\draw [thin] (5,-0.5) -- (4.5,-0.8);
 	\draw [thin] (4.5,-0.8) -- (4.2,-0.5);
 	\draw [thin] (4.2,-0.5) -- (4.2,-1.4);
 	\draw [thin] (4.2,-1.4) -- (4.5,-0.8);

 	\draw [thin] (6,0) -- (6,-1.2);

	\draw [dotted] (2.1,0.8) -- (3,1.2);

	\draw [dotted] (3,-2) -- (4.2,-1.4);

	\draw [fill=white] (1,0.5) circle [radius = 2.5pt];
	\draw [fill=white] (0.8,1) circle [radius = 2.5pt];
	\draw [fill=white] (1.4,1.3) circle [radius = 2.5pt];
	\draw [fill=white] (1.2,1.7) circle [radius = 2.5pt];

	\draw [fill=white] (1.8,0.4) circle [radius = 2.5pt];	
	\draw [fill=white] (2.1,0.8) circle [radius = 2.5pt];	

	\draw [fill=white] (3.1,0.6) circle [radius = 2.5pt];
	\draw [fill=white] (3,1.2) circle [radius = 2.5pt];

	\draw [fill=white] (2.4,-1.3) circle [radius = 2.5pt];
	\draw [fill=white] (1.8,-1.3) circle [radius = 2.5pt];
	\draw [fill=white] (1.1,-1.2) circle [radius = 2.5pt];
	\draw [fill=white] (2,-0.8) circle [radius = 2.5pt];

	\draw [fill=white] (4.8,-1) circle [radius = 2.5pt];
	\draw [fill=white] (4.5,-0.8) circle [radius = 2.5pt];
	\draw [fill=white] (4.2,-0.5) circle [radius = 2.5pt];
	\draw [fill=white] (4.2,-1.4) circle [radius = 2.5pt];

	\draw [fill=white] (6,-1.2) circle [radius = 2.5pt];


	\draw (1.2,1.7) node[above] {$x$};
	\draw (1,0) node[below] {$\pi(x)$};

	\draw (6,-1.2) node[below] {$y$};
	\draw (6,0) node[above] {$\pi(y)$};

      \end{tikzpicture}
      \end{center}
\end{figure}
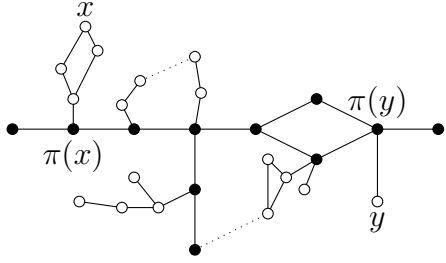

Spinal graphs may be constructed by gluing a collection of finite graphs along another graph; this is made precise in the following example.
In fact, this construction yields all spinal graphs (up to isomorphism, in the usual graph-theoretical sense), as will be shown in Proposition \ref{prop:spinal-glue}.
              
\begin{example}\label{eg:finite-glued}
  Let $\Gamma$ be a connected graph and let $(G_x)_{x \in V(\Gamma)}$ be a collection of finite connected graphs indexed by the vertices of $\Gamma$.
  Suppose that for each $x \in V(\Gamma)$ a distinguished vertex $z_x \in V(G_x)$ is given.
  Then one can construct a graph $G$ by gluing each $G_x$ to $\Gamma$ with the identification $z_x \sim x$.
  More precisely we have
  \begin{equation}\label{eqn:disjtunion}
    V(G) := \bigsqcup_{x \in V(\Gamma)} V(G_x) = \{(x,z) : x \in V(\Gamma), z \in G_x\},
  \end{equation}
  and two vertices $(x_1,y_1)$, $(x_2,y_2)$ are neighbours if and only if either
  \begin{equation*}
    \text{$x_1 = x_2 = x$ and $y_1 \sim y_2$ in $G_x$},
  \end{equation*}
  or
  \begin{equation*}
    \text{$y_1 = z_{x_1}$ and $y_2 = z_{x_2}$ and $x_1 \sim x_2$ in $\Gamma$.}
  \end{equation*}
  The graph $\Gamma$, along with the graphs $(G_x)_{x \in V(\Gamma)}$, naturally embed into $G$.
  We set $\Sigma := V(\Gamma)$ in this embedding; in the disjoint union representation \eqref{eqn:disjtunion} we have
  \begin{equation*}
    \Sigma := \{(x,z_x) : x \in V(\Gamma)\}.
  \end{equation*}
  Every vertex $z \in V(G)$ belongs to precisely one of the embedded graphs $G_x$ with $x \in V(\Gamma)$, and we define $\map{\pi}{V(G)}{\Sigma}$ by the relation $z \in G_{\pi(z)}$; in the disjoint union representation \eqref{eqn:disjtunion}, we have $\pi(x,z) = (x,z_x)$.

  It is immediate that $\pi(x) = x$ for all $x \in \Sigma$, and that each $\pi^{-1}(x)$ is finite.
  Now suppose $a,b \in V(G)$ with $\pi(a) \neq \pi(b)$.
  By construction, every path including $a$ that does not pass through $\pi(a)$ must be entirely contained in $G_{\pi(a)}$.
  Since $b \notin G_{\pi(a)}$, every path from $a$ to $b$ must pass through $\pi(a)$, and by symmetry such a path must also pass through $\pi(b)$.
  That is, every path from $a$ to $b$ contains a subpath from $\pi(a)$ to $\pi(b)$.
  Therefore $(G,\Sigma,\pi)$ is indeed a spinal graph.
\end{example}

\begin{proposition}\label{prop:spinal-glue}
  Suppose $(G,\Sigma,\pi)$ is a spinal graph.
  Then there exists a connected graph $\Gamma$ and a collection $(G_x)_{x \in V(\Gamma)}$ of finite connected graphs such that there exists an isomorphism between $G$ and the graph constructed in Example \ref{eg:finite-glued}, under which $\Sigma$ corresponds to $V(\Gamma)$ and $\pi$ corresponds to the map $\pi^\prime$ given by the relation $z \in G_{\pi'(z)}$. 
\end{proposition}

\begin{proof}
  Let $(G,\Sigma,\pi)$ be a spinal graph.
  Let $\Gamma$ be the full subgraph determined by $\Sigma$, and for every $x \in \Sigma = V(\Gamma)$ let $G_x$ be the full subgraph determined by $\pi^{-1}(x)$.
  Then $\Gamma$ is connected, each $G_x$ is finite and connected (by Lemma \ref{lem:neighbour-lemma}), and we have a bijection
  \begin{equation*}
    \map{\varphi}{V(G)}{\bigsqcup_{x \in V(\Gamma)} V(G_x)}, \quad a \mapsto (\pi(a),a).
  \end{equation*}
  By the construction in Example \ref{eg:finite-glued}, it suffices to show that $a,b \in V(G)$ are neighbours if and only if either $\pi(a) = \pi(b)$ and $a \sim b$ in $G_{\pi(a)}$, or $a = \pi(a)$ and $b = \pi(b)$ and $\pi(a) \sim \pi(b)$ in $\Gamma$.
  
  By Lemma \ref{lem:neighbour-lemma}, if $\pi(a) = \pi(b)$ then every shortest path from $a$ to $b$ is entirely contained in $G_{\pi(a)}$, so in this case $a$ and $b$ are neighbours in $G$ if and only if $a \sim b$ in $G_{\pi(a)}$.
  On the other hand, if $\pi(a) \neq \pi(b)$, then every path from $a$ to $b$ contains a subpath from $\pi(a)$ to $\pi(b)$, and thus $a$ and $b$ are neighbours if and only if either $a = \pi(a)$ and $b = \pi(b)$ or $a = \pi(b)$ and $b = \pi(a)$.
  In the first case, since $\Gamma$ is the full subgraph determined by $\Sigma = \pi(V(G))$, we have $\pi(a) \sim \pi(b)$ in $\Gamma$, and we are done.
  The second case never occurs: since $\pi(b) \in \Sigma$, we would have $\pi(a) = \pi(\pi(b)) = \pi(b)$, which is a contradiction.  
\end{proof}

\section{Dimensions of a spinal graph}

For a spinal graph $(G,\Sigma,\pi)$ we write $d_\Sigma$ and $B_\Sigma$ for the combinatorial distance and balls in the full subgraph determined by $\Sigma$.

\begin{definition}
  Let $(G,\Sigma,\pi)$ be a spinal graph.
  For all $x,y \in V(G)$ define the \emph{spinal distance} $[x,y]$ by
  \begin{equation*}
    [x,y] := d_\Sigma(\pi(x),\pi(y)),
  \end{equation*}
  and for $r > 0$ we define associated \emph{spinal sets} by
  \begin{equation*}
    D(x,r) := \{y \in G : [x,y] \leq r\} = \pi^{-1}(B_\Sigma(\pi(x),r)).
  \end{equation*}
\end{definition}

The spinal distance is a pseudometric on $V(G)$, and the quotient metric space is isometric to $(\Sigma,d_\Sigma)$, but we will not use this fact in what follows.

\begin{definition}\label{dfn:spinal-dim}
  Let $\delta_\Sigma, \delta_G \geq 1$.
  We say that a spinal graph $(G,\Sigma,\pi)$ has \emph{dimensions $(\delta_\Sigma, \delta_G)$} if there exists a point $x_0 \in \Sigma$ and an increasing sequence $(n_k)_{k \in \NN}$ of natural numbers such that for all $k \in \NN$,
  \begin{align}
    |D(x_0,{2n_k})| &\lesssim |D(x_0,n_k)| \label{line:doubling}, \\
    |B_\Sigma(x_0,2n_k)| &\lesssim n_k^{\delta_\Sigma}, \label{line:GVU}\\
    |D(x_0,n_k)| &\simeq n_k^{\delta_G} \label{eqn:GammaVE}.
  \end{align}
\end{definition}

Note that the dimenions of a spinal graph need not be uniquely determined, and may vary for different choices of $x_0$ and $(n_k)_{k \in \NN}$.

\begin{example}\label{eg:vicsek}
  Let $n \in \NN$ and consider the Vicsek graph $\mathcal{V}^n$ in $\RR^n$, the construction of which is given in \cite[Proof of Theorem 4.1]{BCG01}, \cite[Chapter 5]{Chen-thesis}, and \cite[\textsection 5]{CCFR17}.
  One can consider $\mathcal{V}^n$ as a graph with $V(\mc{V}^n) \subset \ZZ^n$, defined as an increasing union of subgraphs $\cup_{m=0}^\infty \mathcal{V}_m^n$.
  The subgraph $\mathcal{V}_0^n$ consists of $2^n + 1$ vertices: one at each corner of the unit $n$-cube, and a central vertex at the origin.
  Each corner vertex is connected to the central vertex.
  For $m \geq 1$, $\mathcal{V}_m^n$ is constructed inductively by connecting a copy of $\mathcal{V}_{m-1}^n$ to each `corner' of $\mathcal{V}_{m-1}^n$.
  It follows that $|V(\mathcal{V}_m^n)| \simeq (2^n + 1)^m$ (see \cite[equation (4.10)]{BCG01}).
  
  Let $\Sigma \in V(\mathcal{V}^n)$ be the set of vertices along the $2^n$ diagonals: with $V(\mc{V}^n) \subset \ZZ^n$, we have
  \begin{equation*}
    \Sigma = \{(\varepsilon_1 m,\varepsilon_2 m,\ldots,\varepsilon_n m) \in \ZZ^n : m \in \NN, \varepsilon_j \in \{1,-1\}\}.
  \end{equation*}
  For every vertex $x \in V(\mathcal{V}^n)$ there is a unique $y \in \Sigma$ such that $x$ and $y$ are connected by a path which intersects $\Sigma$ only at $y$.
  Setting $\pi(x) := y$ makes $(\mathcal{V}^n,\Sigma)$ a spinal graph.
  Pictured in Figure \ref{fig:vicsek} are the first few steps of the construction of $\mathcal{V}^2$, with the spine $\Sigma$ emphasised.
  
  Let $o \in V(\mathcal{V})$ be the `center vertex' of $\mathcal{V}_0^n$, and for $k \in \NN$ let $n_k := 3^k$.
  Then $D(o,n_k) = \mathcal{V}_k^n$, and so
  \begin{equation*}
  	|D(o,n_k)| = (2^n + 1)^k = n_k^{\log_3(2^n + 1)}.
  \end{equation*}
  We also have
  \begin{equation*}
  	|D(o,2n_k)| \leq |D(o,n_{k+1})| = n_{k+1}^{\log_3(2^n + 1)} \simeq n_k^{\log_3(2^n + 1)} = |D(o,n_k)|
  \end{equation*}
  and
  \begin{equation*}
  	|B_\Sigma (o, 2n_k)| = 2^n(2n_k) - 1 < 2^{n+1} 3^k \simeq n_k,
  \end{equation*}
  which tells us that the spinal graph $(\mc{V}^n,\Sigma)$ has dimensions $(1,\log_3(2^n + 1))$.
  In addition, $\mathcal{V}^n$ has polynomial volume growth of dimension $\log_3(2^n + 1)$, that is
  \begin{equation*}
  	|B_\mathcal{V}(x,r)| \simeq r^{\log_3(2^n + 1)}
  \end{equation*}
  for all $x \in V(\mathcal{V}^n)$ and $r \in \NN$.
  (see \cite[page 632]{BCG01}).
  The lower estimate will allow us to apply Corollary \ref{cor:main} to $\mathcal{V}^n$.
  
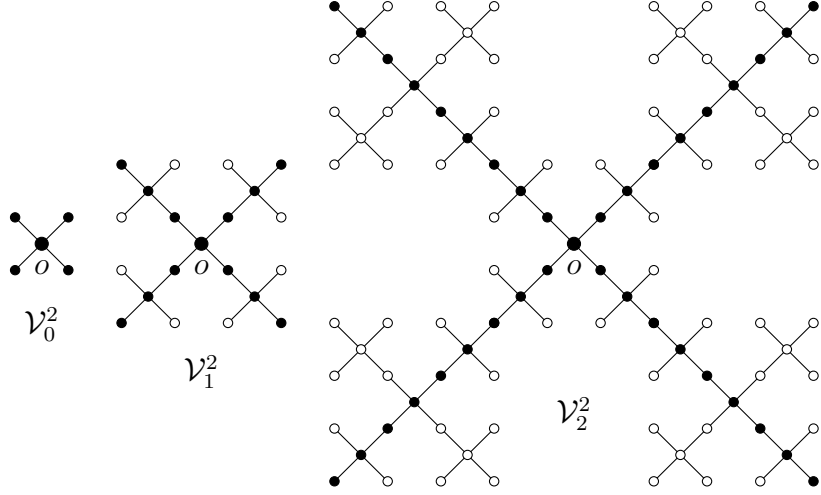
\begin{figure}
  \caption{The first three steps of the construction of the Vicsek graph $\mc{V}^2$, with spine.}
  \label{fig:vicsek}
\begin{center}
  \begin{tikzpicture}[scale=0.7]

    \draw [thin] (0,0) -- (1,1);
    \draw [thin] (0,1) -- (1,0);
    3.5pt
    \draw [fill=black] (0,0) circle [radius = 2.5pt];
    \draw [fill=black] (0,1) circle [radius = 2.5pt];
    \draw [fill=black] (1,0) circle [radius = 2.5pt];
    \draw [fill=black] (1,1) circle [radius = 2.5pt];
    \draw [fill=black] (0.5,0.5) circle [radius = 3.5pt];
    
    \draw (0.5,-0.5) node[below] {$\mathcal{V}_0^2$};
    \draw (0.5,0.4) node[below] {$o$};

    \draw [thin] (2,-1) -- (5,2);
    \draw [thin] (2,2) -- (5,-1);

    \draw [thin] (2,1) -- (3,2);
    \draw [thin] (2,0) -- (3,-1);
    \draw [thin] (4,2) -- (5,1);
    \draw [thin] (4,-1) -- (5,0);
    
    \draw [fill=black] (3,0) circle [radius = 2.5pt];
    \draw [fill=black] (3,1) circle [radius = 2.5pt];
    \draw [fill=black] (4,0) circle [radius = 2.5pt];
    \draw [fill=black] (4,1) circle [radius = 2.5pt];
    \draw [fill=black] (3.5,0.5) circle [radius = 3.5pt];

    \draw [fill=black] (2,-1) circle [radius = 2.5pt];
    \draw [fill=white] (2,0) circle [radius = 2.5pt];
    \draw [fill=white] (3,-1) circle [radius = 2.5pt];
    \draw [fill=black] (2.5,-0.5) circle [radius = 2.5pt];

    \draw [fill=white] (2,1) circle [radius = 2.5pt];
    \draw [fill=black] (2,2) circle [radius = 2.5pt];
    \draw [fill=white] (3,2) circle [radius = 2.5pt];
    \draw [fill=black] (2.5,1.5) circle [radius = 2.5pt];

    \draw [fill=white] (4,2) circle [radius = 2.5pt];
    \draw [fill=white] (5,1) circle [radius = 2.5pt];
    \draw [fill=black] (5,2) circle [radius = 2.5pt];
    \draw [fill=black] (4.5,1.5) circle [radius = 2.5pt];

    \draw [fill=white] (4,-1) circle [radius = 2.5pt];
    \draw [fill=black] (5,-1) circle [radius = 2.5pt];
    \draw [fill=white] (5,0) circle [radius = 2.5pt];
    \draw [fill=black] (4.5,-0.5) circle [radius = 2.5pt];
    
    \draw (3.5,-1.4) node[below] {$\mathcal{V}_1^2$};
    \draw (3.5,0.4) node[below] {$o$};


    \draw [thin] (9,-1) -- (12,2);
    \draw [thin] (9,2) -- (12,-1);

    \draw [thin] (9,1) -- (10,2);
    \draw [thin] (9,0) -- (10,-1);
    \draw [thin] (11,2) -- (12,1);
    \draw [thin] (11,-1) -- (12,0);

    \draw [thin] (6,2) -- (9,5);
    \draw [thin] (6,5) -- (9,2);

    \draw [thin] (6,4) -- (7,5);
    \draw [thin] (6,3) -- (7,2);
    \draw [thin] (8,5) -- (9,4);
    \draw [thin] (8,2) -- (9,3);

    \draw [thin] (6,-4) -- (9,-1);
    \draw [thin] (6,-1) -- (9,-4);

    \draw [thin] (6,-2) -- (7,-1);
    \draw [thin] (6,-3) -- (7,-4);
    \draw [thin] (8,-1) -- (9,-2);
    \draw [thin] (8,-4) -- (9,-3);

    \draw [thin] (12,2) -- (15,5);
    \draw [thin] (12,5) -- (15,2);

    \draw [thin] (12,4) -- (13,5);
    \draw [thin] (12,3) -- (13,2);
    \draw [thin] (14,5) -- (15,4);
    \draw [thin] (14,2) -- (15,3);

    \draw [thin] (12,-4) -- (15,-1);
    \draw [thin] (12,-1) -- (15,-4);

    \draw [thin] (12,-2) -- (13,-1);
    \draw [thin] (12,-3) -- (13,-4);
    \draw [thin] (14,-1) -- (15,-2);
    \draw [thin] (14,-4) -- (15,-3);
    
    \draw [fill=black] (10,0) circle [radius = 2.5pt];
    \draw [fill=black] (10,1) circle [radius = 2.5pt];
    \draw [fill=black] (11,0) circle [radius = 2.5pt];
    \draw [fill=black] (11,1) circle [radius = 2.5pt];
    \draw [fill=black] (10.5,0.5) circle [radius = 3.5pt];

    \draw [fill=black] (9,-1) circle [radius = 2.5pt];
    \draw [fill=white] (9,0) circle [radius = 2.5pt];
    \draw [fill=white] (10,-1) circle [radius = 2.5pt];
    \draw [fill=black] (9.5,-0.5) circle [radius = 2.5pt];

    \draw [fill=white] (9,1) circle [radius = 2.5pt];
    \draw [fill=black] (9,2) circle [radius = 2.5pt];
    \draw [fill=white] (10,2) circle [radius = 2.5pt];
    \draw [fill=black] (9.5,1.5) circle [radius = 2.5pt];

    \draw [fill=white] (11,2) circle [radius = 2.5pt];
    \draw [fill=white] (12,1) circle [radius = 2.5pt];
    \draw [fill=black] (12,2) circle [radius = 2.5pt];
    \draw [fill=black] (11.5,1.5) circle [radius = 2.5pt];

    \draw [fill=white] (11,-1) circle [radius = 2.5pt];
    \draw [fill=black] (12,-1) circle [radius = 2.5pt];
    \draw [fill=white] (12,0) circle [radius = 2.5pt];
    \draw [fill=black] (11.5,-0.5) circle [radius = 2.5pt];

    \draw [fill=white] (7,3) circle [radius = 2.5pt];
    \draw [fill=black] (7,4) circle [radius = 2.5pt];
    \draw [fill=black] (8,3) circle [radius = 2.5pt];
    \draw [fill=white] (8,4) circle [radius = 2.5pt];
    \draw [fill=black] (7.5,3.5) circle [radius = 2.5pt];

    \draw [fill=white] (7,2) circle [radius = 2.5pt];
    \draw [fill=white] (6,2) circle [radius = 2.5pt];
    \draw [fill=white] (6,3) circle [radius = 2.5pt];
    \draw [fill=white] (6.5,2.5) circle [radius = 2.5pt];
    
    \draw [fill=white] (8,2) circle [radius = 2.5pt];
    \draw [fill=white] (9,3) circle [radius = 2.5pt];
    \draw [fill=black] (8.5,2.5) circle [radius = 2.5pt];

    \draw [fill=white] (6,4) circle [radius = 2.5pt];
    \draw [fill=black] (6,5) circle [radius = 2.5pt];
    \draw [fill=white] (7,5) circle [radius = 2.5pt];
    \draw [fill=black] (6.5,4.5) circle [radius = 2.5pt];

    \draw [fill=white] (8,5) circle [radius = 2.5pt];
    \draw [fill=white] (9,4) circle [radius = 2.5pt];
    \draw [fill=white] (9,5) circle [radius = 2.5pt];
    \draw [fill=white] (8.5,4.5) circle [radius = 2.5pt];

    \draw [fill=black] (7,-3) circle [radius = 2.5pt];
    \draw [fill=white] (7,-2) circle [radius = 2.5pt];
    \draw [fill=white] (8,-3) circle [radius = 2.5pt];
    \draw [fill=black] (8,-2) circle [radius = 2.5pt];
    \draw [fill=black] (7.5,-2.5) circle [radius = 2.5pt];

    \draw [fill=white] (7,-4) circle [radius = 2.5pt];
    \draw [fill=black] (6,-4) circle [radius = 2.5pt];
    \draw [fill=white] (6,-3) circle [radius = 2.5pt];
    \draw [fill=black] (6.5,-3.5) circle [radius = 2.5pt];
    
    \draw [fill=white] (8,-4) circle [radius = 2.5pt];
    \draw [fill=white] (9,-3) circle [radius = 2.5pt];
    \draw [fill=white] (9,-4) circle [radius = 2.5pt];
    \draw [fill=white] (8.5,-3.5) circle [radius = 2.5pt];

    \draw [fill=white] (6,-2) circle [radius = 2.5pt];
    \draw [fill=white] (6,-1) circle [radius = 2.5pt];
    \draw [fill=white] (7,-1) circle [radius = 2.5pt];
    \draw [fill=white] (6.5,-1.5) circle [radius = 2.5pt];

    \draw [fill=white] (8,-1) circle [radius = 2.5pt];
    \draw [fill=white] (9,-2) circle [radius = 2.5pt];
    \draw [fill=black] (9,-1) circle [radius = 2.5pt];
    \draw [fill=black] (8.5,-1.5) circle [radius = 2.5pt];

    \draw [fill=black] (13,3) circle [radius = 2.5pt];
    \draw [fill=white] (13,4) circle [radius = 2.5pt];
    \draw [fill=white] (14,3) circle [radius = 2.5pt];
    \draw [fill=black] (14,4) circle [radius = 2.5pt];
    \draw [fill=black] (13.5,3.5) circle [radius = 2.5pt];

    \draw [fill=white] (13,2) circle [radius = 2.5pt];
    \draw [fill=black] (12,2) circle [radius = 2.5pt];
    \draw [fill=white] (12,3) circle [radius = 2.5pt];
    \draw [fill=black] (12.5,2.5) circle [radius = 2.5pt];
    
    \draw [fill=white] (14,2) circle [radius = 2.5pt];
    \draw [fill=white] (15,3) circle [radius = 2.5pt];
    \draw [fill=white] (14.5,2.5) circle [radius = 2.5pt];
    \draw [fill=white] (15,2) circle [radius = 2.5pt];

    \draw [fill=white] (12,4) circle [radius = 2.5pt];
    \draw [fill=white] (12,5) circle [radius = 2.5pt];
    \draw [fill=white] (13,5) circle [radius = 2.5pt];
    \draw [fill=white] (12.5,4.5) circle [radius = 2.5pt];

    \draw [fill=white] (14,5) circle [radius = 2.5pt];
    \draw [fill=white] (15,4) circle [radius = 2.5pt];
    \draw [fill=black] (15,5) circle [radius = 2.5pt];
    \draw [fill=black] (14.5,4.5) circle [radius = 2.5pt];
    
    \draw [fill=white] (13,-3) circle [radius = 2.5pt];
    \draw [fill=black] (13,-2) circle [radius = 2.5pt];
    \draw [fill=black] (14,-3) circle [radius = 2.5pt];
    \draw [fill=white] (14,-2) circle [radius = 2.5pt];
    \draw [fill=black] (13.5,-2.5) circle [radius = 2.5pt];

    \draw [fill=white] (13,-4) circle [radius = 2.5pt];
    \draw [fill=white] (12,-4) circle [radius = 2.5pt];
    \draw [fill=white] (12,-3) circle [radius = 2.5pt];
    \draw [fill=white] (12.5,-3.5) circle [radius = 2.5pt];
    
    \draw [fill=white] (14,-4) circle [radius = 2.5pt];
    \draw [fill=white] (15,-3) circle [radius = 2.5pt];
    \draw [fill=black] (15,-4) circle [radius = 2.5pt];
    \draw [fill=black] (14.5,-3.5) circle [radius = 2.5pt];

    \draw [fill=white] (12,-2) circle [radius = 2.5pt];
    \draw [fill=black] (12,-1) circle [radius = 2.5pt];
    \draw [fill=white] (13,-1) circle [radius = 2.5pt];
    \draw [fill=black] (12.5,-1.5) circle [radius = 2.5pt];

    \draw [fill=white] (14,-1) circle [radius = 2.5pt];
    \draw [fill=white] (15,-2) circle [radius = 2.5pt];
    \draw [fill=white] (15,-1) circle [radius = 2.5pt];
    \draw [fill=white] (14.5,-1.5) circle [radius = 2.5pt];
    
    \draw (10.5,-2.2) node[below] {$\mathcal{V}_2^2$};
    \draw (10.5,0.4) node[below] {$o$};
    
  \end{tikzpicture}
\end{center}
\end{figure}

\end{example}

In Section \ref{sec:examples} we construct spinal graphs with global volume lower bounds and dimensions $(1,D)$ with $D > 1$ that do not arise from fractals.

\section{Nash-type inequalities and spinal dimensional consequences}

Now we assume that $G$ is locally finite.
For each vertex $x \in V(G)$ let $m(x) < \infty$ denote the number of neighbours of $x$, and for each $\map{f}{V(G)}{\CC}$ define the length of the gradient $|\nabla f(x)|$ by
\begin{equation*}
	|\nabla f(x)| := \bigg( \frac{1}{2} \sum_{\substack{y \in V(G) \\ y \sim x}} \frac{1}{m(x)} |f(y) - f(x)|^2 \bigg)^{1/2}.
\end{equation*}

For $1 < p \leq \infty$ and $\beta > 0$, we consider the Nash-type inequality
\begin{equation*}
	S(p,\beta): \qquad \norm{f}_p^{1 + \frac{p^\prime}{\beta}} \lesssim \norm{f}_1^{\frac{p^\prime}{\beta }} \norm{ |\nabla f| }_p \qquad (\text{$\map{f}{V(G)}{\CC}$ finitely supported}),
\end{equation*}
which $G$ may or may not satisfy.

In the presence of a spinal structure, the inequality $S(p,\beta)$ gives quantitative information connecting the `spinal volume growth' of $G$ with the volume growth of $\Sigma$.
This is shown by constructing test functions, defined in terms of the spinal distance, which are constant on the fibres $\pi^{-1}(x)$.
The gradients of these test functions are supported on the spine $\Sigma$, while the functions themselves are supported on spinal sets.

\begin{lemma}\label{lem:volume-est}
  Let $(G,\Sigma,\pi)$ be a spinal graph.
  Fix $p \in (1,\infty)$ and suppose that $G$ satisfies $S(p,\beta)$.
  Then for every $x_0 \in \Sigma$ and $n \in \NN$ we have
  \begin{equation}\label{eqn:pre-Spb}
    |D(x_0,n)|^{\frac{1}{p} \left( 1 + \frac{p^\prime}{\beta} \right)} \lesssim n^{-1} |D(x_0,2n)|^{\frac{p^\prime}{\beta}} |B_\Sigma(x_0,2n)|^{\frac{1}{p}}. 
  \end{equation}
\end{lemma}

\begin{proof}
  For each $x_0 \in \Sigma$ and $n \in \NN$ define $\map{g_n}{V(G)}{[0,1]}$ by
  \begin{equation*}
    g_{n}(x) := \frac{\max(0,n-[x,x_0])}{n}.
  \end{equation*}
  Note that $g_{n}(x) = 0$ if and only if $[x,x_0] \geq n$, so that $\supp g_{n} = D(x_0,n-1)$.
  Furthermore note that $g_{n}$ is constant on each $\pi^{-1}(x)$.
	
  Since $|g_{2n}| \leq 1$ and $\supp g_{2n} \subset D(x_0,2n)$ we have
  \begin{equation}\label{1UB}
    \norm{g_{2n}}_1 \lesssim |D(x_0,2n)|.
  \end{equation}
  Next, since $g_{2n}(x) \geq 1/2$ for $x \in D(x_0,n)$, we have
  \begin{equation}\label{pLB}
    \norm{ g_{2n} }_p \geq \bigg( \sum_{x \in D(x_0,n)} 2^{-p} \bigg)^{1/p} \simeq |D(x_0,n)|^{1/p}.
  \end{equation}
  Finally, note that $|\nabla g_{2n}(x)|= 0$ whenever $x \in G \sm \Sigma$ (since $g_{2n}$ is constant on each connected component of $G \sm \Sigma$) or $x \in G \sm D(x_0,2n)$ (since $\supp g_{2n} = D(x_0,2n-1)$).
	When $x \in \Sigma \cap D(x_0,2n)$ and $y \sim x$, we have
	\begin{equation*}
		g_{2n}(x) - g_{2n}(y) = \left\{ \begin{array}{ll} (2n)^{-1} & \text{if $y \in \Sigma \cap D(x_0,2n)$} \\ 0 & \text{otherwise.} \end{array} \right.
	\end{equation*}
	Therefore
	\begin{align*}
		\norm{ |\nabla g_{2n}| }_p
		&= \bigg( \sum_{x \in \Sigma \cap D(x_0,2n)} \bigg( \frac{1}{2} \sum_{y \sim x, y \in \Sigma \cap D(x_0,2n)} \frac{1}{m(x)} (2n)^{-2} \bigg)^{p/2} \bigg)^{1/p} \\
		&\lesssim n^{-1} |\Sigma \cap D(x_0,2n)|^{1/p} = n^{-1} |B_\Sigma(x_0,2n)|^{1/p}
	\end{align*}
	Therefore, applying $S(p,\beta)$ to $g_{2n}$, we get \eqref{eqn:pre-Spb}.
      \end{proof}

The previous lemma can be used to show that the Nash-type inequalities $S(p,\beta)$ restrict the possible dimensions of a spinal graph.

\begin{lemma}\label{lem:mainlem}
  Let $(G,\Sigma,\pi)$ be a spinal graph with dimensions $(\delta_\Sigma, \delta_G)$.
  Fix $p > 1$ and $\beta > 0$, and suppose $G$ satisfies $S(p,\beta)$.		
  Then
  \begin{equation}\label{eqn:dim-ineq}
    \frac{\delta_G - \delta_\Sigma}{p} - \frac{\delta_G}{\beta} \leq -1.
  \end{equation}
\end{lemma}

\begin{proof}
  Fix $x_0 \in \Sigma$ and a sequence $(n_k)_{k \in \NN}$ as in Definition \ref{dfn:spinal-dim}.
  From Lemma \ref{lem:volume-est} and assumptions \eqref{line:GVU} and \eqref{line:doubling}, for all $k \in \NN$ we have
  \begin{align*}
    |D(x_0,n_k)|^{\frac{1}{p} \left( 1 + \frac{p^\prime}{\beta}\right)} &\lesssim 
                                                                          n_k^{-1} |D(x_0,2n_k)|^{\frac{p^\prime}{\beta}}  |B_\Sigma(x_0,2n_k)|^{\frac{1}{p}} \\
    &\lesssim  n_k^{-1 + \frac{\delta_\Sigma}{p}} |D(x_0,n_k)|^{\frac{p^\prime}{\beta}}. 
  \end{align*}
  Rearranging yields
  \begin{equation*}
    |D(x_0,n_k)|^{\frac{1}{p} - \frac{1}{\beta}} \lesssim n_k^{-1 + \frac{\delta_\Sigma}{p}},
  \end{equation*}
  and then
  \begin{equation*}
    n_k^{\delta_G\left(\frac{1}{p} - \frac{1}{\beta}\right)} \lesssim n_k^{-1 + \frac{\delta_\Sigma}{p}}.
  \end{equation*}
  follows by \eqref{eqn:GammaVE}.
  Since $n_k$ is increasing, taking the limit as $k \to \infty$ tells us that
  \begin{equation*}
    \delta_G\left(\frac{1}{p} - \frac{1}{\beta}\right) \leq -1 + \frac{\delta_\Sigma}{p},
  \end{equation*}
  which is equivalent to \eqref{eqn:dim-ineq}.
\end{proof}

\begin{corollary}\label{cor:mainp}
	Suppose the conditions of Lemma \ref{lem:mainlem} are satisfied, with $\delta_G > \delta_\Sigma$.
	Then $\delta_G > \beta$ and
	\begin{equation}\label{eqn:mainp}
		p \geq \beta \frac{\delta_G - \delta_\Sigma}{\delta_G - \beta}.
	\end{equation}
\end{corollary}

\begin{proof}
	Rearranging \eqref{eqn:dim-ineq} gives
	\begin{equation*}
		\delta_G \geq \beta\left(1 + \frac{\delta_G - \delta_\Sigma}{p}\right).
	\end{equation*}
	Since $\delta_G - \delta_\Sigma > 0$, we get $\delta_G > \beta$.
	We can then rearrange further to get \eqref{eqn:mainp}.
\end{proof}

\section{Riesz transform unboundedness for thickened spinal graphs}

\begin{definition}
  Let $G$ be a uniformly locally finite graph (i.e. $\sup_{x \in V(G)} m(x) < \infty$) and $n \in \NN$.
  Then an $n$-dimensional \emph{thickening} of $G$ is a smooth Riemannian manifold $M$ constructed by replacing each vertex $x \in V(G)$ by an $n$-sphere, each edge $e \in E(G)$ by an $n$-cylinder, and welding the cylinders smoothly to the balls according to the graph structure of $G$, in such a way that $M$ has bounded geometry (i.e. $M$ has positive injectivity radius, and Ricci curvature bounded from below).
\end{definition}

More precisely: define
\begin{equation*}
  \wtd{M} := \bigsqcup_{x \in V(G)} B_x \sqcup \bigsqcup_{e \in E(G)} C_e,
\end{equation*}
where $B_x$ is isometric to a round $n$-sphere $S^n$ with $m(x)$ disjoint open balls of fixed small radius $\varepsilon$ removed, and where each $C_e$ is isometric to a cylinder $S_\varepsilon^{n-1} \times [0,1]$, with $S_\varepsilon^{n-1} = \partial B(0,\varepsilon) \subset \RR^n$.
A $C^0$ Riemannian manifold $M^\prime$ is constructed as a quotient of $\wtd{M}$ by gluing a cylinder $C_e$ to two spheres $B_x$ and $B_y$ if and only if $x \sim y$ in $G$ (in such a way that every `hole' in $B_x$ has a cylinder attached to it).
A thickening $M$ with bounded geometry may then be defined by smoothing the interface between spheres and the cylinders in $M^\prime$ arbitrarily (but uniformly among all the interfaces).

\begin{remark}
  In what follows, we may replace a thickening of $G$ (in the sense above) by any Riemannian manifold $M$ of bounded geometry that is \emph{isometric to $G$ at infinity} in the sense of Coulhon--Saloff-Coste \cite{CSC95} (following Kanai \cite{mK86}); our discretisation/thickening procedures only depend on results in \cite{CSC95}.
\end{remark}
	
The following proposition can be proven by directly following the proof of \cite[Proposition 6.2]{CD03} (see also the first part of \cite[Theorem 5.1]{CCFR17}).
The proof involves the discretisation results of \cite[\textsection 6]{CSC95}.

\begin{proposition}\label{prop:CCFR}
	Let $G$ be a locally uniformly finite graph, and let $M$ be a thickening of $G$ of any dimension.
	Fix $p \in (1,\infty)$ and suppose that $M$ satisfies $(RR_p)$.
	Furthermore, suppose that the heat kernel $h$ of $M$ satisfies
	\begin{equation*}
		h_t(x,x) \lesssim t^{-\alpha/2} \qquad \text{for all $t > 1$, $x \in M$}.
	\end{equation*}
	Then $G$ satisfies $S(p,\alpha)$.
\end{proposition}

%

Since $S(p,\alpha)$ restricts the possible dimensions of a spinal graph, we may argue by contraposition to show that dimension and volume information on a spinal graph implies unboundedness of the Riesz transform on $L^p(M)$ for sufficiently large $p > 2$ (in fact, we prove that $(RR_{p'})$ does not hold for sufficiently large $p > 2$, which is strictly stronger). 

\begin{theorem}\label{thm:main2}
  Let $(G,\Sigma,\pi)$ be a locally uniformly finite spinal graph with dimensions $(\delta_\Sigma, \delta_G)$, with $\delta_G > \delta_\Sigma$.
  Furthermore, suppose that $B_G(x,r) \gtrsim r^\nu$ for all $r \geq 1$.
  Let $M$ be a thickening of $G$ of any dimension.
  Then for all
  \begin{equation*}
    p > 2\frac{\delta_G-\delta_\Sigma}{\frac{\delta_G}{\nu^\prime} - 2\delta_\Sigma + 2} =: p_c(\delta_\Sigma, \delta_G, \nu),
  \end{equation*}
  $M$ does not satisfy $(R_p)$.
\end{theorem}

\begin{proof}
	The volume assumption on $G$ implies a corresponding large-scale volume estimate
	\begin{equation*}
		V(x,r) \gtrsim r^{\nu} \qquad (\text{for all $r > 1$, $x \in M$})
	\end{equation*}
	on $M$ (this may be derived from the results of \cite[\textsection 6]{CSC95}).
	Since $M$ has bounded geometry, \cite[Theorem 1.1]{BCG01} implies the heat kernel estimate
	\begin{equation*}
		h_t(x,x) \lesssim t^{-\nu/(\nu + 1)} \qquad \text{for all $t > 1$}
	\end{equation*}
	on $M$.
	Fix $q > 1$ and suppose that $M$ satisfies $(R_q)$, hence also $(RR_{q^\prime})$.
	Proposition \ref{prop:CCFR} then implies that $G$ satisfies $S(q^\prime, 2\nu/(\nu + 1))$,
	and Corollary \ref{cor:mainp} then yields
	\begin{equation*}
		q^\prime \geq \frac{2\nu}{\nu+1} \left( \frac{\delta_G - \delta_\Sigma}{\delta_G - \frac{2\nu}{\nu+1}} \right) = p_c^\prime
	\end{equation*}
	or equivalently that $q \leq p_c$.
	Therefore $M$ does not satisfy $(R_p)$ for any $p > p_c$.
\end{proof}

Taking $\delta_\Sigma = 1$ and $\nu = \delta_G$ gives the following corollary.

\begin{corollary}\label{cor:main}
  Let $(G,\Sigma,\pi)$ be a locally uniformly finite spinal graph with dimensions $(1,\delta_G)$, with $\delta_G > 1$, and suppose that $B_G(x,r) \gtrsim r^{\delta_G}$ for all $r \in \NN$.
  Let $M$ be a thickening of $G$ of any dimension.
  Then the Riesz transform bound $(R_p)$ for $M$ fails for all $p > 2$.
\end{corollary}

As remarked in Example \ref{eg:vicsek}, these assumptions are satisfied by the Vicsek graphs $\mc{V}^n$, reproving \cite[Theorem 5.1]{CCFR17}. 

\section{Non-fractal spinal graphs with volume lower bounds}\label{sec:examples}

Fix $D > 1$.
In this section we show how to construct locally uniformly finite spinal graphs $(G,\Sigma,\pi)$ with dimensions $(1,D)$ and the volume lower bound
\begin{equation}\label{eqn:VLB}
  |B_G(x,r)| \gtrsim r^D \qquad (x \in V(G), r \in \NN),
\end{equation}
but which need not possess any `fractal' structure (in contrast with the Vicsek graph example).
Corollary \ref{cor:main} applies to such spinal graphs, thus yielding many manifolds $M$ for which $(R_p)$ fails for all $p > 2$.

First we need a technical lemma on volumes of intersections of balls in doubling graphs.
We defer the proof to Section \ref{sec:lemprf}.
Recall that a graph $G$ is \emph{doubling} if there exist constants $C_d,\nu > 0$ such that for all $0 < r < R < \infty$ and  $x \in V(G)$,
\begin{equation*}
  |B_G(x,R)| \leq C_d(R/r)^\nu |B_G(x,r)|.
\end{equation*}
Taking $R = 1$ and $r = 1-\varepsilon$ for $\varepsilon$ arbitrarily small shows that a doubling graph is locally uniformly finite, with $m(x) \leq C_d$ for all $x \in V(G)$.
	
\begin{lemma}\label{lem:edge-regularity}
  Let $G$ be a doubling graph.
  Then there exists $C > 0$, depending only on the doubling constants of $G$, such that for all $y \in V(G)$ and $R > 0$, and for all $x \in B(y,R)$ and $r \leq 2R$, we have
  \begin{equation*}
    |B_G(x,r) \cap B_G(y,R)| \geq C|B_G(x,r)|.
  \end{equation*}
\end{lemma}
	
Now we move on to our construction.
This is inspired by the `plate' construction in \cite[Theorem 5.1]{BCG01}.

\begin{example}\label{eg:construction}
Fix $\delta > D$, and let $(P_n)_{n \in \NN}$ be a family of graphs satisfying
\begin{equation*}
  |B_{P_n}(x,r)| \simeq r^\delta \qquad (x \in V(P_n), r \in \NN)
\end{equation*}
with implicit constants independent of $n$.
For simplicity one can take each $P_n$ to be equal to a fixed graph $P$; one could even take $\delta \in \NN$ and $P = \ZZ^\delta$.
We allow for arbitrary choices to emphasise that self-similarity is not necessary.
Let $\alpha = (D-1)/\delta$ (so that $\alpha \delta + 1 = D$ and $\alpha < 1$) and for each $n \in \NN$ choose an arbitrary vertex $o_n \in V(P_n)$.
Construct a spinal graph $(G,\Sigma,\pi)$ with $\Sigma = \NN$ as in Example \ref{eg:finite-glued}  by taking $G_n$ to be the full subgraph of $P_n$ determined by $B_{P_n}(o_n,n^\alpha)$, and choosing as distinguished points $z_n = o_n$.
	
To show that this spinal graph has dimensions $(1,D)$, take the sequence $n_k = k$ and observe that
\begin{equation*}
  |D(1,k)| = \sum_{n=1}^k |B_{P_n}(o_n,n^\alpha)| \simeq \sum_{n=1}^k n^{\alpha \delta} \simeq k^{\alpha \delta + 1} = k^D
\end{equation*}
(the second sum may be estimated by comparing with integrals of the function $t \mapsto t^{\alpha\delta}$).
In particular we have $|D(1,2k)| \simeq (2k)^D \simeq |D(1,k)|$, and furthermore it is clear that $|B_\NN(x,r)| \simeq r$.
Therefore the spinal graph has dimensions $(1,D)$.
        
It is more difficult to show the global volume lower bound \eqref{eqn:VLB}, but luckily the proof of \cite[Theorem 5.1]{BCG01} already does this for a similar problem.
Note that it suffices to assume $r \geq 2$. 

First we show that $|B_G(n,r)| \gtrsim r^D$ for all $n \in \NN$.
To see this, write
\begin{align*}
  |B_G(n,r)| &\geq \sum_{k=0}^{r} |B_{P_{n+k}}(o_{n+k},\min(r-k,(n+k)^\alpha))| \\
             &\geq \sum_{k=0}^{\floor{r/2}} |B_{P_{n+k}}(o_{n+k},\min(r-k,k^\alpha))| \\
             &\simeq \sum_{k=0}^{\floor{r/2}} k^{\alpha \delta}
               \simeq \floor{r/2}^{\alpha \delta + 1}
               \simeq r^D,
\end{align*}
using that $k^\alpha < r-k$ for $k \leq r/2$ in the third line.
	
Now suppose $x \in V(G)$ with $\pi(x) = n$.
After identifying $x$ with the appropriate vertex $z \in B_{P_n}(o_n, n^\alpha)$ (which, recall, is identified with $\pi(x)$), we get an identification of $B_G(x,r) \cap \pi^{-1}(n)$ with $B_{P_n}(z,r) \cap B_{P_n}(o_n, n^\alpha)$.
Thus for $r \leq 2n^\alpha$ we have
\begin{align*}
  |B_G(x,r)| &\geq |B_G(x,r) \cap \pi^{-1}(n)| \\
             &= |B_{P_n}(z,r) \cap B_{P_n}(o_n,n^\alpha)| \\
             &\gtrsim |B_{P_n}(z,r)|
               \simeq r^\delta 
               > r^D
\end{align*}
using Lemma \ref{lem:edge-regularity} in the third line.
On the other hand, if $r > 2n^\alpha$, then $B_G(x,r)$ contains both $\pi^{-1}(n)$ and $B_G(n+1,r-1-n^\alpha)$, so 
\begin{align*}
  |B_G(x,r)| &\geq \max(|\pi^{-1}(n)|, |B_G(n+1,r-1-n^\alpha)|) \\
             &\gtrsim |\pi^{-1}(n)| + |B_G(n+1, r-1-n^\alpha)| \\
             &\gtrsim n^{\alpha D} + (r-1-n^\alpha)^D \\
             &\simeq (n^\alpha + r - 1 - n^\alpha)^D \\
             &\simeq r^D.
\end{align*}
This completes the proof of \eqref{eqn:VLB}.
        
\end{example}

The following corollary is then an immediate consequence of Corollary \ref{cor:main}.
        
\begin{corollary}
  Suppose $M$ is a thickening of a spinal graph $(G,\Sigma,\pi)$ as constructed as in Example \ref{eg:construction}.
  Then the Riesz transform bound $(R_p)$ for $M$ fails for all $p > 2$.
\end{corollary}

\begin{remark}
  It is probably possible to construct spinal graphs with dimensions $(\delta_\Sigma, \delta_G)$ with $1 < \delta_\Sigma < \delta_G$ and with a polynomial volume lower bound of dimension $\delta_G$, thus yielding manifolds $M$ for which $(R_p)$ fails for all $p > p_c > 2$.
  Since our construction exploits taking $\Sigma = \NN$, this is beyond the scope of this article.
  It may even be possible to show that $(R_p)$ holds on such manifolds for $p \in (2,p_c)$, but this is very much beyond the scope of this article.
\end{remark}

      \section{Proof of Lemma \ref{lem:edge-regularity}}\label{sec:lemprf}

      Here we prove the technical lemma needed in the construction of the previous section.
      We write $B(x,r) := B_G(x,r)$ and $d(x,y) := d_G(x,y)$.

      \begin{proof}
        First note that if the result is true for $r \leq R$, then it holds for $r \leq 2R$, because in this case for $r > R$ we have
        \begin{align*}
          |B(x,r) \cap B(y,R)| \geq |B(x,R) \cap B(y,R)| &\geq C|B(x,R)|
          \\
          &\geq CC_d^{-1}(R/r)^\nu |B(x,r)| \\
          &\geq CC_d^{-1}2^{-\nu} |B(x,r)|
        \end{align*}
        using the doubling condition and $R \geq r/2$ in the last step.
        Thus we assume that $r \leq R$, and split the proof into two cases.
        
        \textbf{Case 1: $r > 2d(x,y)$.}
        By definition of the combinatorial distance, there exists a vertex $z$ such that $d(y,z) + d(z,x) = d(y,x)$ and $d(y,z) = \ceil{d(y,x)/2}$.
        Suppose that $z^\prime \in B(z,r-d(y,z))$.
        Then
        \begin{align*}
          d(z^\prime,x) &\leq r - d(y,z) + d(z,x) \\
			&= r - 2\ceil{d(y,x)/2} + d(y,x) \\
			&\leq r - 2d(y,x) \leq r
        \end{align*}
        and
        \begin{equation*}
          d(z^\prime,y) \leq r - d(y,z) + d(z,y) = r \leq R,
        \end{equation*}
        so $B(z,r-d(y,z)) \subset B(x,r) \cap B(y,R)$.
        Therefore
        \begin{align*}
          |B(x,r) \cap B(y,R)| &\geq |B(z,r-d(y,z))| \\
                               &\gtrsim \left( \frac{r-d(y,z)}{r+d(x,z)} \right)^\nu |B(z,r+d(x,z))| \\
                               &\geq \left( \frac{r-d(y,z)}{r+d(x,z)} \right)^\nu |B(x,r)|
        \end{align*}
        for some $\nu > 0$ determined by the doubling constant of $G$.
        To see that the bracketed expression above is uniformly bounded below, estimate its reciprocal from above by
        \begin{align*}
          \frac{r+d(x,z)}{r-d(y,z)} &= \frac{r-d(y,z)+d(x,y)}{r-d(y,z)} \\
                                    &= 1 + \frac{d(x,y)}{r - d(y,z)} \\
                                    &< 1 + \frac{d(x,y)}{2d(x,y) - d(y,z)} \\
                                    &= 1 + \frac{d(x,y)}{d(x,y) + d(z,x)} \\
                                    &\leq 2.
        \end{align*}
        using $r > 2d(x,y)$ in the third line and $d(z,x) = d(x,y) - d(y,z)$ in the fourth line.
		
        \textbf{Case 2: $r \leq 2d(x,y)$}.
        Note that the estimate for $r < 2$ follows from the fact that $G$ is locally uniformly finite, so it suffices to consider $r \geq 2$.
        As in the previous case, there exists a vertex $z$ such that $d(y,z) + d(z,x) = d(y,x)$ and $d(x,z) = \floor{r/2}$ (here we use that $r/2 \leq d(x,y)$).
        If $z^\prime \in B(z,\floor{r/2})$, then
		\begin{equation*}
			d(z^\prime,x) < 2\left\lfloor \frac{r}{2} \right\rfloor \leq r
		\end{equation*}
		and
		\begin{equation*}
                  d(z^\prime,y) < \left\lfloor \frac{r}{2} \right\rfloor + d(z,y) = \left\lfloor\frac{r}{2}\right\rfloor + d(x,y) - \left\lfloor \frac{r}{2} \right\rfloor = d(x,y) \leq R,
		\end{equation*}
                using that $x \in B(y,R)$ by assumption, and so $B(z,\floor{r/2}) \subset B(x,r) \cap B(y,R)$.
		Therefore
		\begin{align*}
                  |B(x,r) \cap B(y,R)| &\geq |B(z,\floor{r/2})| \\
                                       &\geq |B(z,r/4)| \\
                                       &\geq (8^\nu C_d)^{-1} |B(z,2r)| \\
                                       &\geq (8^\nu C_d)^{-1} |B(z,r+d(z,x))| \\
                                       &\geq (8^\nu C_d)^{-1} |B(x,r)|
		\end{align*}
		since $\floor{r/2} \geq r/4$ holds whenever $r \geq 2$, and using the doubling condition.
              \end{proof}

              \section*{Acknowledgements}
                
              I thank Li Chen and Thierry Coulhon for many interesting discussions on this material.
              I also thank the anonymous referee for their careful reading and suggestions.
                Portions of this work were carried out while the author was a postdoc at Universit\'e Grenoble--Alpes and at TU Delft (supported by the VIDI subsidy 639.032.427 of the Netherlands Organisation for Scientific Research (NWO)), and also while the author held an Australian Mathematical Society Lift-off Fellowship.
        
            \bibliographystyle{amsplain}
\bibliography{riesz}  

\end{document}